\documentclass{birkjour}
 \newtheorem{thm}{Theorem}[section]
 
 \newtheorem{lem}[thm]{Lemma}
 
 \theoremstyle{definition}
 \newtheorem{defn}[thm]{Definition}
 \theoremstyle{remark}
 \newtheorem{rem}[thm]{Remark}
 
 \numberwithin{equation}{section}

\usepackage{enumerate}

\newcommand{\thmref}[1]{Theorem~\ref{#1}}
\newcommand{\lemref}[1]{Lemma~\ref{#1}}

\newcommand{\defnref}[1]{Definition~\ref{#1}}

\renewcommand{\Re}{\operatorname{Re}}
\renewcommand{\Im}{\operatorname{Im}}

\begin{document}

%
%
%
%
%
%
%
%
%

\title[Newton--Hurwitz--Sierpinski]
{The Newton approximation, the Hurwitz continued fraction, and the Sierpinski series for relatively quadratic units over certain imaginary quadratic number fields}

\author[A. Saito]{Asaki Saito}
\address{Department of Complex and Intelligent Systems\\
Future University Hakodate\\
116-2 Kamedanakano-cho, Hakodate, Hokkaido 041-8655\\
Japan}
\email{saito@fun.ac.jp}

\author[J.-I. Tamura]{Jun-Ichi Tamura}
\address{Institute for Mathematics and Computer Science\\
Tsuda College\\
2-1-1 Tsuda-machi, Kodaira, Tokyo 187-8577\\
Japan}
\email{jtamura@tsuda.ac.jp}




\begin{abstract}
The objective of this paper is to show (a)=(b)=(c) as
rational functions of $T$, $U$ for (a), (b), (c) given by
(a) continued fractions of length $2^{n+1}-1$ with explicit partial
denominators in $\left\{-T,U^{-1}T\right\}$,
(b) truncated series $\sum_{0\le m\le n} \left(U^{2^m}/\left(h_0(T)h_1(T,U) \cdots h_m(T,U)\right)\right)$
with $h_n$ defined by $h_0:=T$ and $h_{n+1}(T,U):=h_n(T,U)^2-2U^{2^n}
(n \geq 0)$,
(c) $(n+1)$-fold iteration $F^{(n+1)}(0)= F^{(n+1)}(0,T,U)$ of
$F(X)= \allowbreak F(X,T,U) \allowbreak :=X-f(X)/\frac{df}{dX}(X)$ for $f(X)=X^2-T X+U$,
and to find explicit equalities among truncated Hurwitz continued
fraction expansion of relatively quadratic units $\alpha \in
\mathbb{C}$ over imaginary quadratic fields
$\mathbb{Q}\left(\sqrt{-1}\right)$, $\mathbb{Q}\left(\sqrt{-3}\right)$,
rapidly convergent complex series called the Sierpinski series, and the Newton approximation of
$\alpha$ on the complex plane.
We also give an estimate of the error of the Newton
approximation of the unit $\alpha$.
\end{abstract}

\maketitle

\section{Introduction}

W. Sierpinski has
found a curious identity
between quadratic irrationals and ascending continued fractions 
\begin{align}\label{eq:RealQuadraticIrrationalAscendingCF}
\frac{t-\sqrt{t^2-4}}{2}=\frac{1+\frac{1+\frac{1+ ^{^{.\cdot{}^{\cdot} }}}{g_2(t)}}{g_1(t)}}{g_0(t)} \left(=\sum_{m\geq
0} \frac{1}{g_0(t)g_1(t)\cdots g_m(t)} \right),
\quad                         t=3,4,5,\ldots ,
\end{align}
where  $g_m(t)$ are polynomials defined inductively by 
\begin{align}\label{eq:Gpolynomials}
g_0(t):=t, \quad  g_{m+1}(t)= P\left( g_m(t)\right) ~(m\ge 0)
\end{align}
with $P(X):=X^2-2$, which generalize the equality \eqref{eq:RealQuadraticIrrationalAscendingCF} with $t=3$:
\[
\frac{3-\sqrt{5}}{2}=\frac{1}{3 }+\frac{1}{3\times 7}+\frac{1}{3\times 7\times 47}+\frac{1}{3\times 7\times 47\times
2207}+\cdots 
\]
found by  E. Lucas, cf.  \cite[p. 331]{Lucas}, \cite[pp. 279-281]{Sierpinski}.  We have given an explicit
equality between 
the truncated ascending continued fraction of \eqref{eq:RealQuadraticIrrationalAscendingCF} and a finite symmetric
simple continued fraction
\begin{align}\label{eq:SymmetricCF}
\begin{split}
& \sum_{0\le m\le n} \frac{1}{g_0(t)g_1(t)\cdots g_m(t)}\\
& =[0;t-1,1,t-2,\ldots,   1,t-2,1,t-1] \quad (t\ge 3,n\ge 1), 
\end{split}
\end{align}
where the block ``$1,t-2$'' is repeated $2^{n+1}-3$ times, giving a total
length of $2^{n+2}-3$,
cf. \cite{Tamura1}. Moreover, under a certain modification of the recurrence \eqref{eq:Gpolynomials} with
$\deg_X P(X)\ge 3$, such an explicit equality can be extended to equalities
between
truncated ascending continued fractions and symmetric continued fractions such
that the limit of the fractions on  both sides tends to
a transcendental number, 
cf. \cite{Tamura2}.

The objective of this paper is to extend the equality \eqref{eq:RealQuadraticIrrationalAscendingCF}
to complex numbers,
to disclose the link between the Newton approximations and continued
 fractions. In
what follows, we mean by $S(t,u)$ a series (an ascending continued fraction) defined by 
\begin{align}\label{eq:SierpinskiSeries}
\begin{split}
 & S(t,u):=\sum_{m\geq 0} \frac{u^{2^m}}{h_0(t)h_1(t,u)\cdots h_m(t,u)} \left(=\frac{u^{2^0}+\frac{u^{2^1}+\frac{u^{2^2}+ ^{^{.\cdot{}^{\cdot} }}}{h_2(t)}}{h_1(t)}}{h_0(t)} \right),\\
&  h_0:=t, \quad  h_{n+1}:=h_n^2-2u^{2^n} (n\geq 0),
\end{split}
\end{align}
which will be referred to as a \textit{Sierpinski}\textit{ }\textit{series},
and by $S_n(t,u)$ its truncation
\begin{align}\label{eq:TruncatedSierpinskiSeries}
S_n(t,u):=\sum_{0\le m\le n} \frac{u^{2^m}}{h_0(t)h_1(t,u)\cdots h_m(t,u)}.
\end{align}
Let $\alpha = \alpha (t,u)$ ($|\alpha|<1$) be a relatively quadratic unit such that 
\begin{align}\label{eq:MinimalPolnomialOverGausianField}
\begin{split}
& f(X)=f(X;t,u):=X^2-t X+u \in R_G[X],\\
& u\in \left\{\pm 1,\pm \sqrt{-1}\right\}, \quad t\in R_G\backslash G_2(u)\subset R_G:=\mathbb{Z}\left[\sqrt{-1}\right],
\end{split}
\end{align}
as its  minimal polynomial over the Gaussian field $K_G:=\mathbb{Q}\left(\sqrt{-1}\right)$, where
$G_2(u) 
\left(u\in \left\{\pm 1,\pm \sqrt{-1}\right\}\right)$ are certain finite subsets of $R_G$, cf.
\lemref{Lem:OpenDiscsIncludingRoots}.  We shall see
that for all $t\in R_G\backslash G_2(u)$, $u\in \left\{\pm 1,\pm \sqrt{-1}\right\}$,
\begin{align}\label{eq:ComplexQuadraticIrrationalAscendingCF}
\alpha =\sum_{m\geq 0} \frac{u^{2^m}}{h_0(t)h_1(t,u)\cdots h_m(t,u)}
\end{align}
holds. The equality \eqref{eq:ComplexQuadraticIrrationalAscendingCF} is an extension of \eqref{eq:RealQuadraticIrrationalAscendingCF} to complex numbers.
In some cases,
the Newton approximation $F^{(n)}(c)$ works even for a complex root $\alpha$ of
an analytic function $f(X)$ with a complex number $c$,
where $F^{(n)}(X)$ is the $n$-fold iteration
of $F(X)
=F(X,f)$ defined by
\[
                       F(X)=F(X;f):=X-f(X)\left/\frac{d f}{d X}(X)\right.,
\]
which will be called as the \textit{Newton iterator} of
$f(X)$.  We show for example,
the equalities among three quantities\\
            (N)   $F^{(n+1)}(0)$ with $F(X)=$ the Newton iterator of $f=X^2-t X+u\in R_G[X]$,\\
            (S)   $S_n(t,u)$,\\
            (H)  ($HCF_G$) $\left[a_0;a_1,a_2,\ldots ,a_{2^{n+1}-1}\right]$, $a_n= a_n(t,u)$\\
for $t\in R_G\backslash G_2(u)$, $u\in \left\{\pm 1,\pm \sqrt{-1}\right\}$
(here, (N), (S), and (H) come from the initials  of Newton, Sierpinski,
and  Hurwitz),
where
\begin{align*}
& (HCF_G) \left[a_0;a_1,a_2,\ldots ,a_n,\ldots \right],\\
& a_0\in R_G, \quad a_n\in R_G\backslash \left\{0,\pm 1,\pm \sqrt{-1}\right\} ~  (n\ge 1)
\end{align*}
is a \textit{Hurwitz continued fraction expansion}
of the root  $\alpha = t\left(1-\sqrt{1-4u/t^2}\right)/2$ of $f(X)$,
cf. \defnref{Defn:PrincipalValueOfSquareRoot} and \thmref{Thm:IdentitiesAsComplexNumbers}.

In Section \ref{HurwitzContinuedFractionExpansion}, we give some lemmas for continued fractions, mainly
concerning the 
Hurwitz continued fractions.  We give identities of rational functions
among the 
Newton iterators, truncated Sierpinski series, and continued fractions in Sections
\ref{NewtonIteratorSierpinskiSeries} and
\ref{NewtonIteratorContinuedFractions}, see Lemmas \ref{Lem:NewtonSierpinski2} and \ref{Lem:NewtonHurwitz}.
We give an estimate of the error of the Newton approximation to the
relatively quadratic units $\alpha$, $\beta$ with $f(X)$ as their
minimal polynomial, for the detail, see Section
\ref{SpeedConvergence}.
Summing up lemmas in Sections \ref{HurwitzContinuedFractionExpansion}--\ref{SpeedConvergence}, we state two theorems concerning identities
among truncations of descending/ascending continued fractions and the
Newton approximation, in Section \ref{MainResults}.
We can do the same for relatively quadratic units over the
Eisenstein 
field $K_E:=\mathbb{Q}\left(\sqrt{-3}\right)$ with
\begin{align*}
& f(X)=f(X;t,u):=X^2-t X+u \in R_E[X],\\
& u \in \left\{z\in \mathbb{C}; z^6=1\right\}, \quad t\in R_E\backslash E_2(u)\subset R_E:=\mathbb{Z}\left[\frac{1+\sqrt{-3}}{2}\right],
\end{align*}
as their minimal polynomials, cf. \thmref{Thm:IdentitiesAsComplexNumbersEisenstein}
in Section \ref{EisensteinField}.
We give
\thmref{Thm:IdentitiesAsComplexNumbersEisenstein} without
a proof, since it is quite parallel to that of
\thmref{Thm:IdentitiesAsComplexNumbers}. We remark 
that Theorems \ref{Thm:IdentitiesAsComplexNumbers}, \ref{Thm:IdentitiesAsComplexNumbersEisenstein} are obtained by identities among rational functions (N), (S),
(H) with two independent variables $T$, $U$ in place of $t$, $u$, cf. \thmref{Thm:IdentitiesAsRationalFunctions}.
Thus, the Hurwitz continued fraction expansions are outstanding algorithm 
since they have simpler\footnote{
The simple continued fraction \eqref{eq:SymmetricCF} is equal to the Hurwitz continued
fraction (H) with $u=1$, 
and the length $2^{n+2}-3$ of the former which is almost twice of
the length $2^{n+1}-1$ of the continued 
fraction (H) so that the continued fraction (H) converges almost two times
faster than the simple
continued fraction \eqref{eq:SymmetricCF}.}
compatibility with such as very rapidly
convergent Newton 
approximations and the Sierpinski series as far as concerning to relatively
quadratic 
units over the fields $K_G$ and $K_E$.

\section{Hurwitz Continued fraction expansion}\label{HurwitzContinuedFractionExpansion}

\begin{lem}\label{Lem:AbsoluteValuesOfRoots}
Let $G_1(u)$ ($u \in \left\{\pm 1,\pm \sqrt{-1}\right\}$) be a subset
of $R_G$
defined by
\begin{align*}
& G_1(1):=\left\{0,\pm 1,\pm 2\right\}, \quad G_1(-1):=\left\{0,\pm \sqrt{-1},\pm 2\sqrt{-1}\right\},\\
& G_1\left(\sqrt{-1}\right):=\left\{0,\pm \left(1+\sqrt{-1}\right)\right\}, \quad G_1\left(-\sqrt{-1}\right):=\left\{0,\pm
\left(1-\sqrt{-1}\right)\right\}.
\end{align*}
Let $f(X)=f(X;t,u):=X^2-t X+u\in R_G[X]$  be a  polynomial with  $t, u$
satisfying 
\[
u \in \left\{\pm 1,\pm \sqrt{-1}\right\}, \quad t \in R_G\backslash G_1(u).
\]
Then $f$ has distinct roots $\alpha$, $\beta$ with $|\alpha|<1<|\beta|$.
\end{lem}

\begin{proof}
Let  $\alpha$, $\beta$ be the roots of $f$. Suppose
that the absolute value of one of the 
roots is one.  Then $|\alpha|=|\beta|=1$ follows from $|\alpha   \beta| =|u|=1$, so
that
$|t|=|\alpha +\beta |\le |\alpha |+|\beta |=2$.  Hence, $|t| >2$ implies that $f$ has distinct
roots
$\alpha, \beta$ with $|\alpha |<1< |\beta  |$.  On the other hand, by direct calculation
one can check
that for  all  $4\times 5^2=100$  pairs of $(u,t)\in R_G^2$ satisfying
\[
u\in \left\{\pm 1,\pm \sqrt{-1}\right\}, \quad -2\le \Re t\le 2, \quad -2\le \Im t\le 2, 
\]
$|\alpha |<1< |\beta |$ does not hold if and only if
\[
u \in \left\{\pm 1,\pm \sqrt{-1}\right\}, \quad t\in G_1(u).
\]
\end{proof}

\begin{defn}\label{Defn:PrincipalValueOfSquareRoot}
We define 
                 $\sqrt{z}:=\sqrt{r}e^{\sqrt{-1}\theta /2 }$ for a complex number $z=re^{\sqrt{-1}\theta}$ ($r>0$, $-\pi <\theta \le \pi$).     
\end{defn}

By \defnref{Defn:PrincipalValueOfSquareRoot}, $-\pi /2<\arg  \sqrt{z}\le \pi /2$ always holds for any $0\neq z\in \mathbb{C}$.
Thus,
the roots $\alpha$, $\beta$ of $f(X)$  explained in \lemref{Lem:AbsoluteValuesOfRoots} are written as
\[
\alpha =\frac{t}{2}\left(1-\sqrt{1-4u/t^2}\right), \quad \beta =\frac{t}{2}\left(1+\sqrt{1-4u/t^2}\right).
\]

We mean by $[x]$ the usual floor function of a real number
$x$. Let
$z=x+\sqrt{-1}y\in \mathbb{C}$ ($x,y \in \mathbb{R}$) be any complex number. An ``integer part''
 $[z]_G$  of  
$z$ is defined by
\[
  [z]_G:=\left[x+\frac{1}{2}\right]+\left[y+\frac{1}{2}\right]\sqrt{-1}\in R_G
\]
which is the lattice point of $R_G$ nearest from $z$ on the complex
plane. We put
\[\langle z \rangle_G
:=z-[z]_G\in \left\{x+\sqrt{-1}y \left| -\frac{1}{2}\leq x<\frac{1}{2}, -\frac{1}{2}\leq y<\frac{1}{2}\right.\right\}
\]
which is a ``fractional part'' of $z$. A.~Hurwitz introduced two continued
fraction expansions for any complex number  $z\in \mathbb{C}$.  By one of
the  expansions
$z$ can be written as
\begin{align*}
& (HCF_G) \left[a_0;a_1,a_2,\ldots ,a_n\right] \quad \textrm{or} \quad (HCF_G) \left[a_0;a_1,a_2,\ldots
\right],\\
& a_0\in R_G, \quad  a_n\in R_G\backslash \left\{0,\pm 1,\pm \sqrt{-1}\right\} ~ (n\ge 1).
\end{align*}
Suppose for simplicity that $z\in \mathbb{C} \backslash K_G$
for a while. Setting
                  $z_0:=z=a_0+\zeta_0$  with $a_0:=\left[z_0\right]_G$, $\zeta_0:=\langle z_0 \rangle_G$,
we have $\zeta_0\in \mathbb{C} \backslash K_G$, so that $\zeta_0\neq 0$. Setting  $z_1:=1/\zeta_0$, we can write
\[
z_0=a_0+\frac{1}{z_1}=a_0+\frac{1}{a_1+\zeta _1}
\]
with $a_1:=\left[z_1\right]_G$, $\zeta _1:=\langle z_1 \rangle_G \in \mathbb{C} \backslash K_G$, $\zeta_1 \neq 0$. Hence setting
$z_2:=1/\zeta_1$, we can write
\[
z_0=a_0+\frac{1}{a_1+\displaystyle\frac{1}{z_2}}
\]
Repeating the procedure, we have an infinite continued fraction for $z \in
\mathbb{C} \backslash K_G$
written as
\[
HCF_G \left[a_0; a_1,a_2,\ldots \right]:=a_0+\frac{1}{ a_1+\displaystyle\frac{1}{a_2+_{ \ddots}} } 
\]
It is easy to see  $a_n \in R_G\backslash \left\{0,\pm 1,\pm \sqrt{-1}\right\}$ ($n\ge 1$). Such
a continued fraction
expansion will be referred to as a \textit{Gauss}\textit{-}\textit{Hurwitz
continued fraction}
(abbr. $HCF_G$) \textit{expansion}. In particular for a real number, the
$HCF_G$ expansion 
turns out to be the so called \textit{nearest}\textit{ }\textit{integer}\textit{
}\textit{continued}\textit{ }\textit{fraction}\textit{ }\textit{expansion}.

We use later the following fundamental properties (H1), (H2), (H3)
of $HCF_G$, cf. \cite{Hurwitz}:   

\begin{enumerate}
\item[(H1)]   The expansion terminates for $z \in K_G$, i.e.,
$z \allowbreak = \allowbreak \left(HCF_G\right) \allowbreak  \left[a_0;a_1,a_2,\ldots ,a_n\right]$ holds for an integer $n \geq 0$,
\item[(H2)]  The expansion converges to $z$ for $z \in \mathbb{C} \backslash K_G$, so that we
can write 
$z= \left(HCF_G\right)\left[a_0; a_1,a_2,\ldots \right]:=\lim_{n\to \infty}  \left(HCF_G\right)\left[a_0;a_1,a_2,\ldots
,a_n\right]$,
\item[(H3)]  The $HCF_G$ expansion of any relatively quadratic algebraic element
over $K_G$
becomes eventually periodic, and vice versa.   
\end{enumerate}

If one replace the definition of  $[z]_G$  ($z\in \mathbb{C}$) by $[z]_E$
as the element 
of the lattice
\[
R_E=\left\{\left. x \times \frac{1+\sqrt{-3}}{2}+y \times \frac{1-\sqrt{-3}}{2} \right|
x,y\in \mathbb{Z}\right\}
\]
nearest from $z$ on the complex plane in the argument above, one can
define a continued 
fraction expansion with partial denominators in $R_E$, to be precise see
\cite{Hurwitz}. Such a 
continued fraction will be written as $\left(HCF_E\right)\left[a_0; a_1,a_2,\ldots \right]$
referred to as an 
\textit{Eisenstein}\textit{-}\textit{Hurwitz}\textit{ }\textit{continued}\textit{ }\textit{fraction} (abbr.,
$HCF_E$).  The $HCF_E$ expansion is also an
extension of the nearest integer continued fraction expansion of real
numbers, and 
$HCF_E$ also has the properties (H1), (H2), (H3).

We denote by $B(c,r)$ an open disc
\[
B(c,r):=\{\left. z \in \mathbb{C} ~\right|~ |z-c| <r\}, \quad c \in \mathbb{C}, \quad r>0.  
\]

\begin{lem}\label{Lem:OpenDiscsIncludingRoots}
Let $D:=\left\{x+\sqrt{-1}y \in R_G ; -1\leq x\leq
1, -1\leq y\leq 1\right\}$, and let  
$G_2 (u)$ ($u \in \left\{\pm 1,\pm \sqrt{-1}\right\}$) be sets given by
$G_2(1):=D \cup \{\pm 2\}$, $G_2\left(\sqrt{-1}\right):=D \cup \left\{\pm \left(1+2\sqrt{-1}\right),\pm \left(2+\sqrt{-1}\right)\right\}$,
$G_2(-1):=D\cup \left\{\pm 2\sqrt{-1}\right\}$, $G_2\left(-\sqrt{-1}\right):=D \cup \left\{\pm \left(1-2\sqrt{-1}\right),\pm \left(2-\sqrt{-1}\right)\right\}$,
see Figure~\ref{fg:G2}.
Let $\alpha$, $\beta$ be the roots of
\[
f=X^2-tX+u, \quad u \in \left\{\pm 1,\pm \sqrt{-1}\right\}, \quad t \in R_G\backslash G_2(u)
\]
as in \lemref{Lem:AbsoluteValuesOfRoots}. Then
\[
\alpha \in B(0,1/2), \quad \beta \in B(t,1/2)
\]
hold.
In particular,  $[\alpha (t,u)]_G=0$, $[\beta (t,u)]_G=t$.      
\end{lem}

\begin{figure}
\begin{center}
\includegraphics[width=0.4\textwidth,clip]{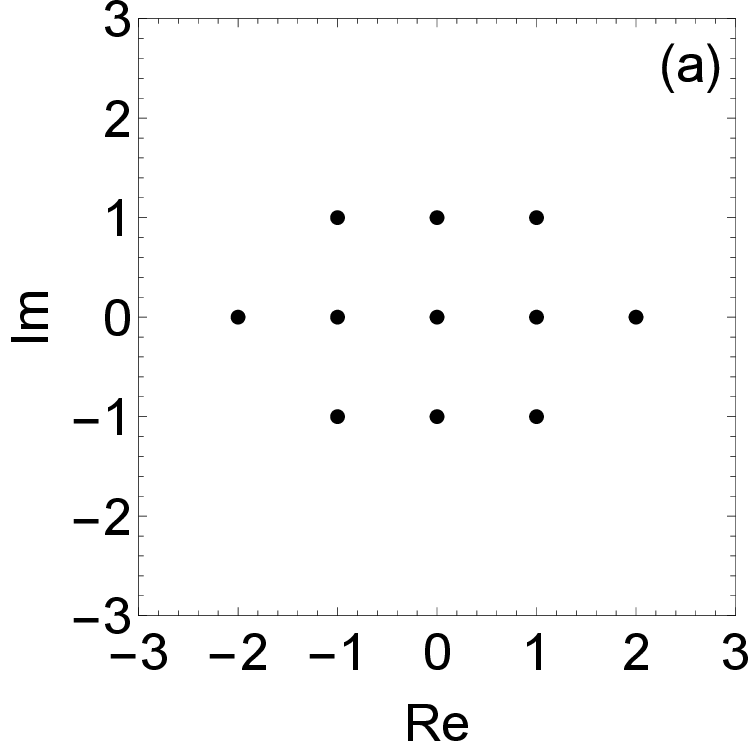}
\hspace*{0.05\textwidth}\includegraphics[width=0.4\textwidth,clip]{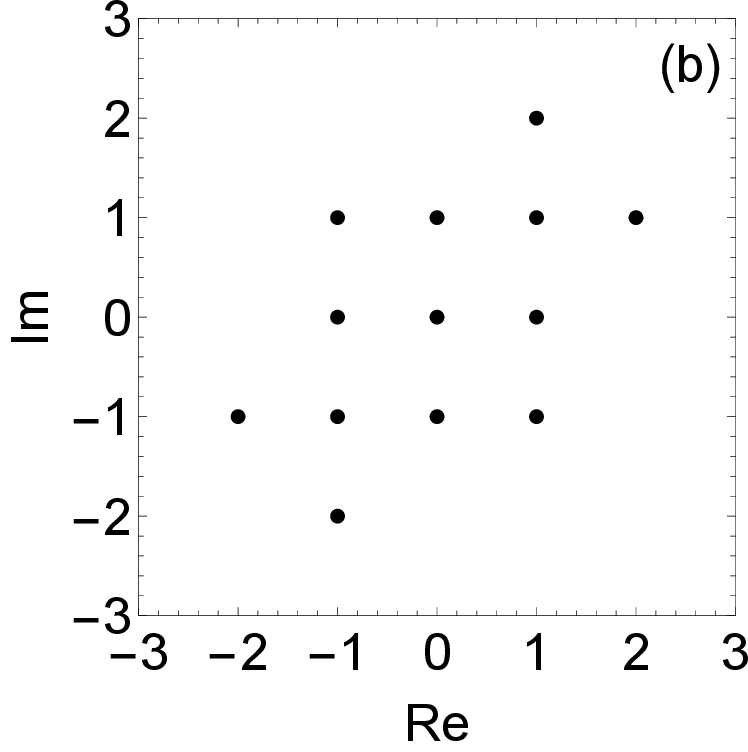}\\
\vspace*{0.025\textwidth}
\includegraphics[width=0.4\textwidth,clip]{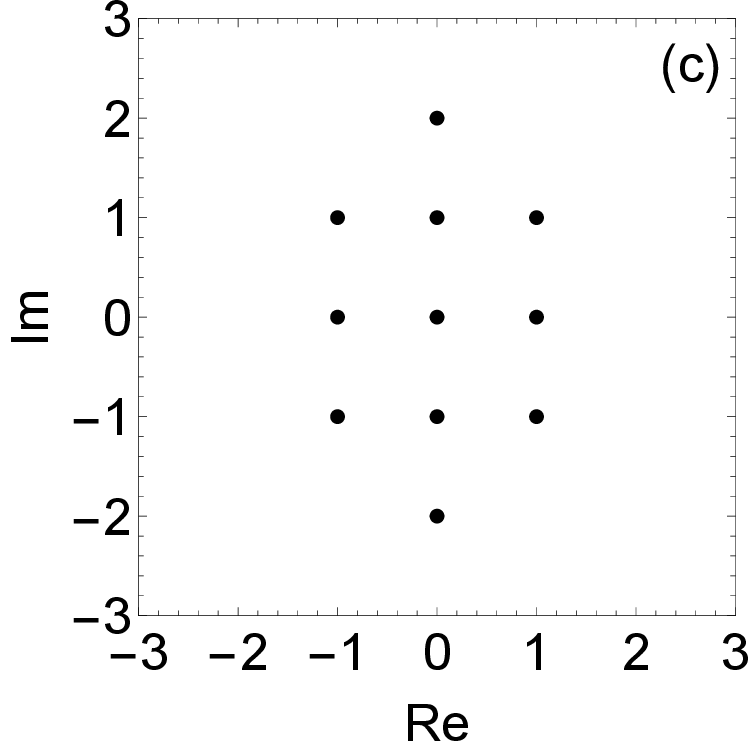}
\hspace*{0.05\textwidth}\includegraphics[width=0.4\textwidth,clip]{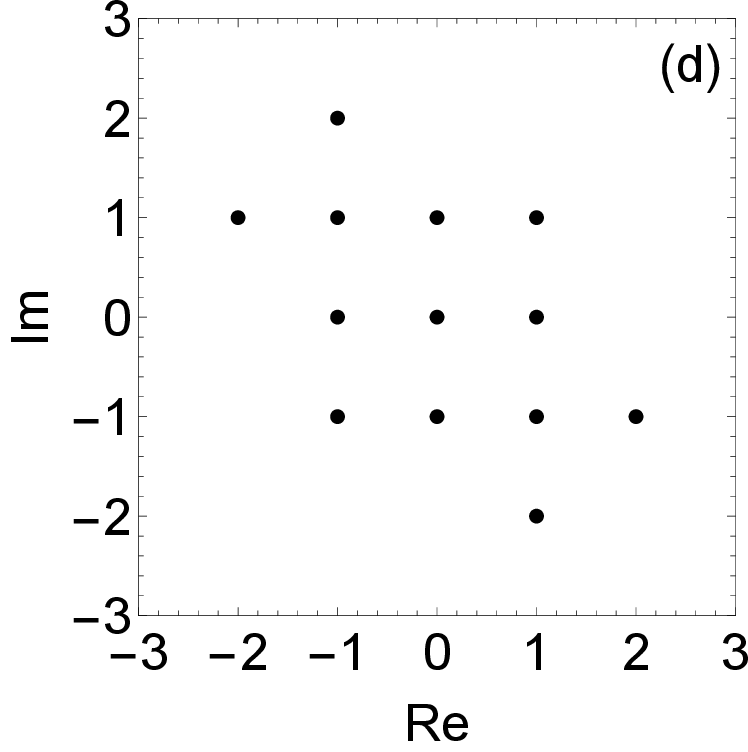}
\end{center}
\caption{\label{fg:G2}
(a) $G_2(1)$,
(b) $G_2\left(\sqrt{-1}\right)$,
(c) $G_2(-1)$, and
(d) $G_2\left(-\sqrt{-1}\right)$.}
\end{figure}

\begin{rem}
Since $G_2(u) \supset G_1(u)$  ($\forall u \in \left\{\pm 1, \pm \sqrt{-1}\right\}$),
$f$ has two roots
 $\alpha$, $\beta$  with  $|\alpha |<1< |\beta |$ by \lemref{Lem:AbsoluteValuesOfRoots}.      
\end{rem}

\begin{proof}[Proof of \lemref{Lem:OpenDiscsIncludingRoots}]
Let $\alpha$, $\beta$ be the roots of
$f$  for $t \in R_G$,
$u \in \left\{\pm 1,\pm \sqrt{-1}\right\}$ in general with $|\alpha |\le |\beta |$.  By  direct
calculation,
one can check that for  $4 \times 5^2=100$  pairs of $(u,t)\in R_G^2$ satisfying
\[
u \in \left\{\pm 1,\pm \sqrt{-1}\right\}, \quad -2\le \Re t\le 2, \quad -2\le \Im t\le 2,
\]
$|\alpha |<1/2$ does not hold if and only if
\[
u \in \left\{\pm 1,\pm \sqrt{-1}\right\}, \quad t \in G_2(u).
\]
Now
suppose that
\[
u \in \left\{\pm 1, \pm \sqrt{-1}\right\}, \quad t=x+y\sqrt{-1} \textrm{~with~} |x| \ge 3 \textrm{~or~} |y| \ge 3.
\]
Then $\left|4u/t^2\right|=\left|4/t^2\right|\le 4/9$, so that $\left|\arg \left(1-4u/t^2\right)\right| \le  \arcsin 4/9$. Recalling
\defnref{Defn:PrincipalValueOfSquareRoot},  we have $\left|\arg  \sqrt{1-4u/t^2} \right|\le (\arcsin  4/9)/2$, and
\begin{align*}
\left|\alpha\right| &=
\left| \frac{t\left(1-\sqrt{1-4u/t^2}\right)}{2} \right|
= \frac{2} {\left|t\right| \left|1+\sqrt{1-4u/t^2}\right|}\\
& \le  \frac{2}{3 \left(1+\sqrt{1-4/9}\right)}
<\frac{1}{2},
\end{align*}
which implies $\alpha \in B(0,1/2)$, and  $\beta =t-\alpha \in B(t,1/2)$.
\end{proof}

\begin{lem}\label{Lem:HCFGExpansions}
Let $\alpha =\alpha (t,u)$, $\beta =\beta (t,u)$ be numbers with
 $t,
u$ as in \lemref{Lem:OpenDiscsIncludingRoots}. Then the $HCF_G$ expansions of $\alpha, \beta$ are given
by
\begin{align*}
\alpha (t,u) &=\left(HCF_G\right)\left[0;u^{-1}t,-t,u^{-1}t,-t,u^{-1}t,\ldots \right],\\
\beta (t,u) &=\left(HCF_G\right)\left[t;-u^{-1}t,t,-u^{-1}t,t,-u^{-1}t,\ldots \right].
\end{align*}
\end{lem}

\begin{rem}
(1)   \lemref{Lem:OpenDiscsIncludingRoots} implies the irreducibility of the polynomial
$f(X;t,u)$ over  the Gaussian field $K_G$ with $t \in R_G\backslash G_2(u)$, $u \in \left\{\pm 1,\pm \sqrt{-1}\right\}$.\\
(2)   The convergence of all the continued
fractions in \lemref{Lem:HCFGExpansions}
 follows from (H2).\\
(3)   The length of the period of the continued fractions in \lemref{Lem:HCFGExpansions} turns out to be one  if $u=-1$.
\end{rem}

\begin{proof}[Proof of \lemref{Lem:HCFGExpansions}]
We consider the case where $u=1$.
 Let  $\beta =\beta (1,t)$ and
 $\beta =\left(HCF_G\right)\left[b_0; b_1,b_2,\ldots \right]$ be the continued fraction obtained by
the Hurwitz
algorithm. By \lemref{Lem:OpenDiscsIncludingRoots}, $[\beta ]_G=t$, i.e., $b_0=t$, and
\[
\beta =b_0+(\beta -t)
= b_0+\frac{1}{\frac{-t\left(1+\sqrt{1-4/t^2}\right)}{2}}.
\]
Here, $\left[\frac{-t\left(1+\sqrt{1-4/t^2}\right)}{2}\right]_G = -\left[\frac{t\left(1+\sqrt{1-4/t^2}\right)}{2}\right]_G=-t$,
so that we get $b_1=-t$, and
\begin{align*}        
\beta =b_0 +\frac{1}{b_1+\frac{1}{\beta}}.       
\end{align*}
Hence we get $\beta =\left(HCF_G\right)[t;-t,t,-t,\ldots ]$.  By \lemref{Lem:OpenDiscsIncludingRoots}, $|\alpha |< 1/2$,
so that $[\alpha ]_G=0$, which together with $\alpha =1/\beta$ implies 
$\alpha =\left(HCF_G\right)[0;t,-t,t,-t,\ldots ]$. It is a routine work to find the 
$HCF_G$ expansions of $\alpha, \beta$  for the cases $u=-1, \pm \sqrt{-1}$, and we skip the detail.
\end{proof}

\section{The Newton iterator and the Sierpinski series}\label{NewtonIteratorSierpinskiSeries}

In Sections \ref{NewtonIteratorSierpinskiSeries}, \ref{NewtonIteratorContinuedFractions}, we mean by $T, U$ independent indeterminates, by $P:=\mathbb{Z}[T,U]$
the ring of polynomials of variables $T, U$, and by $C:=\mathbb{Q}(T,U)$
the rational
function field.

\begin{lem}\label{Lem:NewtonSierpinski1}
Let $\left\{h_n\right\}_{n\geq 0}$ be a sequence in the
ring $P$ defined by
\[
h_0=T, \quad h_{n+1}=h_n^2-2U^{2^n} ~ (\forall n\geq 0).      
\]
Let $F(X) \in C(X)$ be the Newton iterator for $f=X^2-TX+U \in C(X)$ in the
formal 
sense, i.e.,
\[
F(X)=F(X;f)= X-f \left/\frac{df}{dX}\right.=\frac{X^2-U}{2X-T}
\]
and let $F^{(n)}(X)\in C(X)$ ($F^{(0)}(X):=X$) be a rational function defined
by the 
$n$-fold iteration of $F(X)$.  Then identities
\begin{align}
F^{(n+1)}(0)-F^{(n)}(0) &=\frac{U^{2^n}}{ h_0 h_1\cdots h_n} \quad (n\geq 0),\label{eq:DiffF0}\\
F^{(n+1)}(T)-F^{(n)}(T) &=-\frac{U^{2^n}}{ h_0 h_1\cdots h_n} \quad (n\geq 0) \label{eq:DiffFT}
\end{align}
hold as elements in the field $C$.
\end{lem}

\begin{proof}
First we prove \eqref{eq:DiffF0}. It is
clear that \eqref{eq:DiffF0} holds for 
$n=0$. Suppose that \eqref{eq:DiffF0} holds for a nonnegative integer $n$. Then we have
\[
\left(F\left(F^{(n)}(0)\right)-F^{(n)}(0)=\right)-\frac{F^{(n)}(0)^2-T F^{(n)}(0)+U}{2 F^{(n)}(0)-T}=\frac{U^{2^n}}{
    h_0 h_1\cdots h_n},
\]
which implies
\begin{align}\label{eq:RHSDiffF0}
\frac{U^{2^{n+1}}}{h_0 h_1\cdots h_n h_{n+1}}=-\frac{F^{(n)}(0)^2-T F^{(n)}(0)+U}{2 F^{(n)}(0)-T}\frac{
U^{2^n}}{h_{n+1}}.
\end{align}
Since $F^{(n+1)}(0)=\frac{F^{(n)}(0)^2-U}{2F^{(n)}(0)-T}$, we also have
\begin{multline}\label{eq:LHSDiffF0}
- \frac{F^{(n+1)}(0)^2-T F^{(n+1)}(0)+U}{2F^{(n+1)}(0)-T}\\
=-\frac{\left(F^{(n)}(0)^2-T F^{(n)}(0)+U\right)^2}{\left(2 F^{(n)}(0)-T\right) \left(2 F^{(n)}(0)^2-2 T F^{(n)}(0)+T^2-2U\right)}.
\end{multline}
To complete the proof of \eqref{eq:DiffF0}, it suffices to show
\begin{align}\label{eq:2ndInduction}
\frac{T^2-4U}{F^{(n)}(0)^2-T F^{(n)}(0)+U}+2=\frac{h_{n+1}}{U^{2^n}} \quad (n\geq 0),
\end{align}
since \eqref{eq:2ndInduction} implies
\[
\frac{F^{(n)}(0)^2-T F^{(n)}(0)+U}{2 F^{(n)}(0)^2-2 T F^{(n)}(0)+T^2-2U}=
\frac{U^{2^n}}{h_{n+1}},
\]
which together 
with \eqref{eq:LHSDiffF0} and \eqref{eq:RHSDiffF0} implies
\begin{align*}
-\frac{F^{(n+1)}(0)^2-T F^{(n+1)}(0)+U}{2F^{(n+1)}(0)-T} &=
-\frac{F^{(n)}(0)^2-T F^{(n)}(0)+U}{2 F^{(n)}(0)-T} \cdot \frac{ U^{2^n}}{h_{n+1}}\\     &=\frac{U^{2^{n+1}}}{ h_0 h_1\cdots
h_{n+1}}.
\end{align*}
We show \eqref{eq:2ndInduction} by induction. It is clear that \eqref{eq:2ndInduction} holds for $n=0$. Suppose
that 
\eqref{eq:2ndInduction} holds for a nonnegative integer $n$. Then we have
\begin{align*}
& \frac{h_{n+2}}{U^{2^{n+1}}}=\left(\frac{h_{n+1}}{U^{2^n}}\right)^2-2
=\left(\frac{T^2-4U}{F^{(n)}(0)^2-T F^{(n)}(0)+U}+2\right)^2-2\\
& =\frac{\left(T^2-4U\right)\left(2 F^{(n)}(0)-T\right)^2}{\left(F^{(n)}(0)^2-T F^{(n)}(0)+U\right)^2}+2
=\frac{T^2-4U}{F^{(n+1)}(0)^2-T F^{(n+1)}(0)+U}+2
\end{align*}
Thus, we have proved \eqref{eq:DiffF0}.
We can do the same for \eqref{eq:DiffFT}.
\end{proof}

Recalling the definition of truncated Sierpinski series \eqref{eq:TruncatedSierpinskiSeries}, let $S_n(T,U)$ be a rational function defined by
\[
S_n(T,U):= \sum_{0\le m\le n} \frac{U^{2^m}}{h_0(T)h_1(T,U)\cdots h_m(T,U)} \in
C.
\]
Considering equalities
\begin{align*}
F^{(n+1)}(0) &=F^{(0)}(0)+ \sum_{0\le m\le n} \left(F^{(m+1)}(0)-F^{(m)}(0)\right)\in C,\\
F^{(n+1)}(T) &=F^{(0)}(T)+ \sum_{0\le m\le n} \left(F^{(m+1)}(T)-F^{(m)}(T)\right)\in C,
\end{align*}
we get the following lemma by \lemref{Lem:NewtonSierpinski1}.

\begin{lem}\label{Lem:NewtonSierpinski2}
Two identities
\[
F^{(n+1)}(0)=S_n(T,U), \quad F^{(n+1)}(T)=T-S_n(T,U)
\]
hold for all $n\ge 0$ as rational functions in the field $C$.
\end{lem}

\section{The Newton iterator and Continued fractions}\label{NewtonIteratorContinuedFractions}

Let  $\left\{p_n\right\}_{n\geq -2}$, $\left\{q_n\right\}_{n\geq -2}$ be sequences of polynomials
defined by the recurrences
\begin{alignat*}{4}
p_{-2} &:=0, \quad p_{-1} &:=1, \quad p_n &:=a_np_{n-1}+p_{n-2} \quad &(n\geq 0),\\
q_{-2} &:=1, \quad q_{-1} &:=0, \quad q_n &:=a_nq_{n-1}+q_{n-2} \quad &(n\geq 0),
\end{alignat*}
where $\left\{a_n\right\}_{n\geq 0}$ is a sequence of indeterminates/variables. Then
identities
\begin{align*}
&\frac{p_n}{ q_n} =\left[a_0;a_1,a_2,\ldots ,a_n\right] \quad (n\geq 0),\\
&p_nq_{n-1}-p_{n-1}q_n =(-1)^{n-1} \quad (n\ge -1),\\
&\frac{q_n}{ q_{n-1}} =\left[a_n;a_{n-1},a_{n-2},\ldots,a_1\right] \quad (n\ge 1)
\end{align*}
hold as rational functions of $n+1$ variables $a_m$ ($0\le m\le n$), which is
valid for any 
choice of $a_0, a_1, a_2, \ldots, a_n \in \mathbb{C}$ as far as $q_n\neq 0$, cf.~\cite[pp. 27-31]{Perron}.
In what follows, such formulae together with a
trivial identity
\[
U\left[a_0;a_1,a_2,\ldots ,a_n\right]=  \left[U a_0;U^{-1}a_1,U a_2,\ldots ,U^{(-1)^n}a_n\right]
\]
will be  used without references.

\begin{lem}\label{Lem:IdentitiesPQRS}
Let  $C$ be the rational function
field of two variables
$T$ and  $U$.  Let  $p_n, q_n, r_n, s_n \in C ~(n\ge -2)$ be
rational 
functions of $T, U$ defined by the recurrences
\begin{alignat*}{3}
p_n &=a_np_{n-1}+p_{n-2} \quad &\left(n\ge 0, ~ p_{-1}:=1, ~ p_{-2}=0\right),\\
q_n &=a_nq_{n-1}+q_{n-2} \quad &\left(n\ge 0, ~ q_{-1}:=0, ~ q_{-2}=1\right);\\
r_n &=b_nr_{n-1}+r_{n-2} \quad &\left(n\ge 0, ~ r_{-1}:=1, ~ r_{-2}=0\right),\\
s_n &=b_ns_{n-1}+s_{n-2} \quad &\left(n\ge 0, ~ s_{-1}:=0, ~ s_{-2}=1\right);
\end{alignat*}
with
\begin{align*}
&a_0=0, \quad a_n=U^{-1}T ~ (n=\mathrm{odd}\ge 1), \quad a_n=-T ~ (n=\mathrm{even}\ge 2),\\
&b_n=-U^{-1}T ~ (n=\mathrm{odd}\ge 1), \quad b_n=T ~ (n=\mathrm{even}\ge 0).    
\end{align*}
Then identities\\
(a)   $p_n=q_{n-1} ~ (n=\mathrm{odd}\ge -1), \quad p_n=-U q_{n-1} ~ (n=\mathrm{even}\ge 0),$\\
(b)   $r_n=s_{n+1} ~ (n=\mathrm{odd}\ge -1), \quad r_n=-U s_{n+1} ~ (n=\mathrm{even}\ge 0),$\\
are valid as rational functions in $C$.
\end{lem}

\begin{proof}
(a):   It is clear that (a) holds for $n=-1, 0$.
Suppose that 
$p_{2k-1}=q_{2k-2}$ and $p_{2k}=-U q_{2k-1}$ hold for an integer $k ~ (\ge 0)$. Then
$p_{2k+1}=U^{-1} T p_{2k}+p_{2k-1}=-Tq_{2k-1}+q_{2k-2}=q_{2k}$, so that $p_{2k+2}=-T p_{2k+1}+p_{2k}
=-T q_{2k}-U q_{2k-1}=-U\left(U^{-1}Tq_{2k}+q_{2k-1}\right)=-U q_{2k+1}$. 
 (b): The same as (a).
\end{proof}

\begin{lem}\label{Lem:NewtonHurwitz}
Let  $F^{(n)}(X)\in C(X)$ be as in \lemref{Lem:NewtonSierpinski1}.
Let $a_n, b_n$ be elements of 
$C$ as in \lemref{Lem:IdentitiesPQRS}.
Then the following two identities hold for all $n\ge 0$.
\begin{align}
F^{(n)}(0) &=\left[a_0;a_1,a_2,\ldots ,a_{2^n-1}\right]\in C,\label{eq:NewtonCF0}\\
F^{(n)}(T) &=\left[b_0;b_1,b_2,\ldots ,b_{2^n-1}\right]\in C.\label{eq:NewtonCFT}
\end{align}
\end{lem}

\begin{proof}
The lemma trivially holds for $n=0$.
Let  $p_n$, $q_n$ be as in \lemref{Lem:IdentitiesPQRS}.   First,
we show 
\begin{align}\label{eq:NewtonPQ}
F\left(\frac{p_{2^n-1}}{q_{2^n-1}}\right)=\frac{p_{2^{n+1}-1}}{q_{2^{n+1}-1}} \left(=\left[a_0;a_1,a_2,\ldots ,a_{2^{n+1}-1}\right]\right) \quad (n \ge 0).
\end{align}
The equality \eqref{eq:NewtonPQ} is clearly valid for $n = 0$.
Suppose $n \ge 1$.
In view of \lemref{Lem:IdentitiesPQRS}, we have an identity, for example for $n=2$,
\begin{align*}
\frac{p_{2^{n+1}-1}}{q_{2^{n+1}-1}} &=\left[0; U^{-1}T ,-T, U^{-1}T , -T, U^{-1}T, -T, U^{-1}T \right]\\
&=\left[0;U^{-1}T,-T,U^{-1}T,-T+\left[0;U^{-1}T,-T,U^{-1}T \right]\right]\\
&=\left[a_0;a_1,a_2,a_3,a_4+\left[a_0;a_1,a_2,a_3\right]\right]\\
&=\left[a_0;a_1,a_2,a_3,a_4+\frac{p_{2^n-1}}{q_{2^n-1}}\right]
\end{align*}
and in general, for $n \ge 1$
\begin{align*}
\frac{p_{2^{n+1}-1}}{q_{2^{n+1}-1}} &=\left[a_0;a_1,a_2,\ldots ,a_{2^n-1},a_{2^n}+\frac{p_{2^n-1}}{q_{2^n-1}}\right]\\
&=\frac{\left(a_{2^n}+\displaystyle\frac{p_{2^n-1}}{q_{2^n-1}}\right)p_{2^n-1}+p_{2^n-2}}{\left(a_{2^n}+\displaystyle\frac{p_{2^n-1}}{q_{2^n-1}}\right)q_{2^n-1}+q_{2^n-2}}
=\frac{\displaystyle\frac{p_{2^n}}{q_{2^n-1}}+\left(\frac{p_{2^n-1}}{q_{2^n-1}}\right)^2}{\displaystyle\frac{q_{2^n}}{q_{2^n-1}}+\frac{p_{2^n-1}}{q_{2^n-1}}},
\end{align*}
which together with the equality  $p_n=-Uq_{n-1}$ ($n=\mathrm{even}\ge 0$) in \lemref{Lem:IdentitiesPQRS} implies
\begin{align}\label{eq:NewtonInt}
\frac{p_{2^{n+1}-1}}{q_{2^{n+1}-1}}
=\left.\left(\left(\frac{p_{2^n-1}}{q_{2^n-1}}\right)^2-U \right)\right/\left(\frac{q_{2^n}}{q_{2^n-1}}+\frac{p_{2^n-1}}{q_{2^n-1}}\right).
\end{align}
On the other hand, we have 
\begin{align}\label{eq:qfrac}
\begin{split}
\frac{q_{2^n}}{q_{2^n-1}} &=\left[a_{2^n};a_{2^n-1},\ldots ,a_1\right]=\left[-T;a_{2^n-1},\ldots ,a_1\right]\\
&=-T+\left[0;a_1,\ldots ,a_{2^n-1}\right]=\frac{p_{2^n-1}}{q_{2^n-1}}-T.
\end{split}
\end{align}
In view of \eqref{eq:NewtonInt}, \eqref{eq:qfrac}, and  $F(X)=\frac{X^2-U}{2 X-T}$, we get
\[
F\left(\frac{p_{2^n-1}}{q_{2^n-1}}\right)=\frac{\left(\frac{p_{2^n-1}}{q_{2^n-1}}\right)^2-U}{2\frac{p_{2^n-1}}{q_{2^n-1}}-T}=\frac{p_{2^{n+1}-1}}{q_{2^{n+1}-1}}.
\]
Thus, we have proved \eqref{eq:NewtonPQ}. Using \eqref{eq:NewtonPQ} repeatedly, we get  
\begin{align*}
&\left[0;a_1(U,T),a_2(U,T),\ldots ,a_{2^n-1}(U,T)\right]\\
&= \frac{p_{2^n-1}}{q_{2^n-1}}=F\left(\frac{p_{2^{n-1}-1}}{q_{2^{n-1}-1}}\right)=F^{(2)}\left(\frac{p_{2^{n-2}-1}}{q_{2^{n-2}-1}}\right)
=\cdots
=F^{(n)}\left(\frac{p_{2^0-1}}{q_{2^0-1}}\right)\\
&=F^{(n)}\left(\frac{p_0}{q_0}\right)=F^{(n)}(0),
\end{align*}
which complete the proof of \eqref{eq:NewtonCF0}.

Secondly, we show
\begin{align}\label{eq:NewtonRS}
F\left(\frac{r_{2^n-1}}{s_{2^n-1}}\right)=\frac{r_{2^{n+1}-1}}{s_{2^{n+1}-1}} \left(=\left[b_0;b_1,b_2,\ldots ,b_{2^{n+1}-1}\right]\right) \quad (n \ge 0),
\end{align}
since \eqref{eq:NewtonRS} implies \eqref{eq:NewtonCFT}.
It is clear for $n = 0$, so that we suppose $n \ge 1$.
For $n=2$
\begin{align*}
\frac{r_{2^{n+1}-1 }}{s_{2^{n+1}-1}} &=\left[T;-U^{-1}T,T,-U^{-1}T,T,-U^{-1}T,T,-U^{-1}T  \right]\\
&=\left[T;-U^{-1}T,T,-U^{-1}T,\left[T;-U^{-1}T,T,-U^{-1}T \right] \right]\\
&=\left[b_0;b_1,b_2,b_3,\frac{r_{2^n-1}}{s_{2^n-1}}\right],
\end{align*}
and in general, for $n \ge 1$
\begin{align*}
\frac{r_{2^{n+1}-1}}{s_{2^{n+1}-1}} &=\left[b_0;b_1,b_2,\ldots ,b_{2^n-1},\frac{r_{2^n-1}}{s_{2^n-1}}\right]\\
&=\frac{\displaystyle\frac{r_{2^n-1}}{s_{2^n-1}}r_{2^n-1}+r_{2^n-2}}{\displaystyle\frac{r_{2^n-1}}{s_{2^n-1}}s_{2^n-1}+s_{2^n-2}}
=\frac{\displaystyle\left(\frac{r_{2^n-1}}{s_{2^n-1}}\right)^2+\frac{r_{2^n-2}}{s_{2^n-1}}}{\displaystyle\frac{r_{2^n-1}}{s_{2^n-1}}+\frac{s_{2^n-2}}{s_{2^n-1}}}
=
\frac{\displaystyle\left(\frac{r_{2^n-1}}{s_{2^n-1}}\right)^2-U}{\displaystyle\frac{r_{2^n-1}}{s_{2^n-1}}+\frac{s_{2^n-2}}{s_{2^n-1}}}.
\end{align*}
Since
\[
\frac{r_{2^n-1}}{s_{2^n-1}}=\left[b_0;b_1,b_2,\ldots ,b_{2^n-1}\right]
\]
with $b_n=T$ ($n=\mathrm{even}\ge 0$), $b_n=-U^{-1}T$ ($n=\mathrm{odd}\ge 1$), we have
\begin{align*}
&\frac{r_{2^n-1}}{s_{2^n-1}}-T\\
&=\left[0;b_1,b_2,\ldots ,b_{2^n-2},b_{2^n-1}\right]
=\left[0;-U^{-1}b_0,-U b_1,\ldots ,-U b_{2^n-3},-U^{-1}b_{2^n-2}\right]\\
&=-U\left[0;b_0,b_1,\ldots , b_{2^n-3},b_{2^n-2}\right]
=-U\left/\frac{r_{2^n-2}}{s_{2^n-2}}\right.=\frac{s_{2^n-2}}{-U^{-1}r_{2^n-2}}=\frac{s_{2^n-2}}{s_{2^n-1}}
\end{align*}
by \lemref{Lem:IdentitiesPQRS}.
Thus, we get 
\[
\frac{r_{2^{n+1-1 }}}{s_{2^{n+1}-1}}=\frac{\left(\frac{r_{2^n-1}}{s_{2^n-1}}\right){}^2-U
  }{2\frac{r_{2^n-1}}{s_{2^n-1}}-T}=F\left(\frac{r_{2^n-1}}{s_{2^n-1}}\right).
\]
\end{proof}

\section{The speed of Convergence}\label{SpeedConvergence}


\begin{lem}\label{Lem:h_n__g_n}
Let  $t, u \in \mathbb{C}$  with $|t| >2$, $|u|=1$.  Let
$\left\{h_n\right\}_{n\ge 0}$ be a sequence of 
complex numbers defined by
\[
h_n:=h_{n-1}^2-2u^{2^{n-1}} \left(n \geq 1, h_0:=t\right).
\]
Let $\left\{g_n\right\}_{n\ge 0}$ be a sequence of real numbers defined by
\[
g_n:=g_{n-1}^2-2 \left(n\geq 1, g_0:=|t|\right).
\]
Then  $\left|h_n\right| >2$, $g_n>2$, and 
\[
\left|h_n\right| \ge g_n=\rho^{2^n}+\frac{1}{\rho^{2^n}}  \left(\rho :=\frac{|t|+\sqrt{|t|^2-4}}{2}>1\right) 
\]
holds for all $n\ge 0$.   
\end{lem}

\begin{proof}
Since $\left|t^2-2u \right| >2$ holds for $t, u \in \mathbb{C}$  with $|t|>2$ and $|u|=1$, we
have $\left|h_n\right| > 2$, $g_n > 2$ ($\forall n\ge 0$).  We show $\left|h_n\right| \ge g_n$
by induction.  It is clear that
$\left|h_0\right| \ge g_0>2$.  Suppose that $|h_n| \ge g_n$ holds for an integer
$n\ge 0$. Then we get
\[
\left| h_{n+1}\right|=\left|h_n^2-2u^{2^n}\right| \geq |h_n|^2 -2\ge g_n^2 -2 =g_{n+1}.
\]
It is easy to see $g_0=|t | = \rho +\frac{1}{\rho}$ and $g_n=\rho^{2^n}+\frac{1}{\rho^{2^n}}$ by
considering that
$\rho$, $\frac{1}{\rho} =\frac{|t|-\sqrt{|t|^2-4}}{2}$ are the two real roots of $X^2-|t |X+1=0$.
\end{proof}

We define $G_3(u) =G_2(u) \cup \left\{\pm 2,\pm 2\sqrt{-1}\right\}$  ($u \in \left\{\pm 1,\pm \sqrt{-1}\right\}$),
see Figure~\ref{fg:G3}.
Then we have the following lemma.

\begin{figure}
\begin{center}
\includegraphics[width=0.4\textwidth,clip]{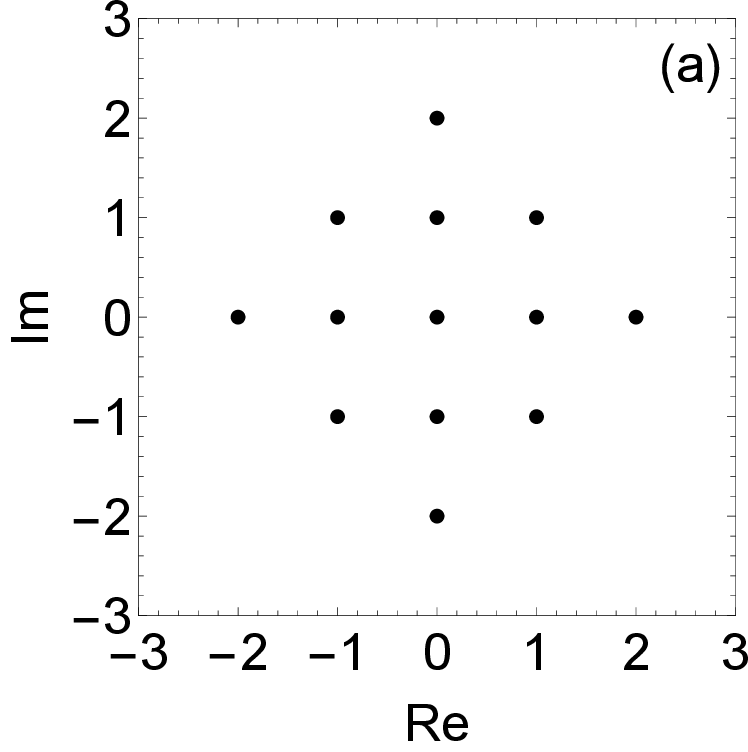}
\hspace*{0.05\textwidth}\includegraphics[width=0.4\textwidth,clip]{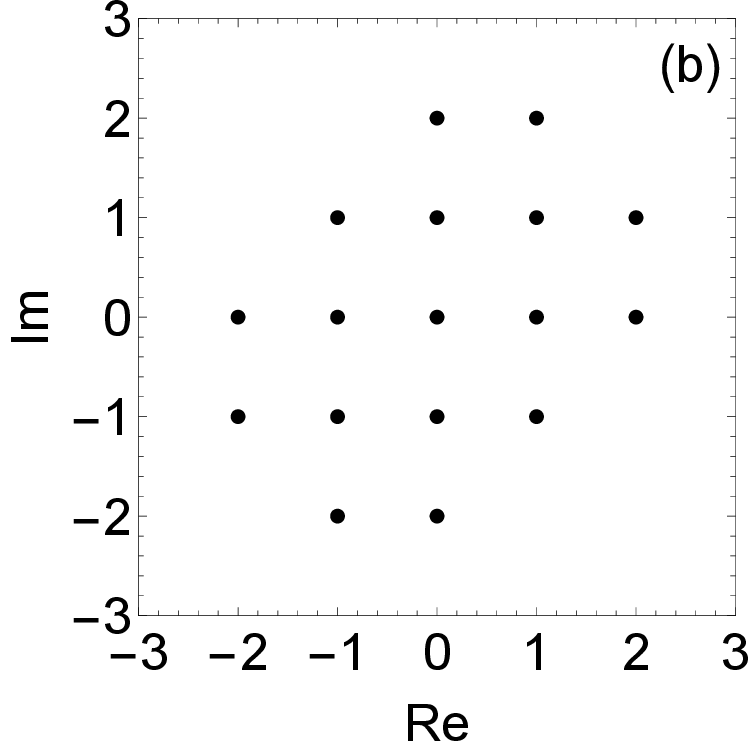}\\
\vspace*{0.025\textwidth}
\includegraphics[width=0.4\textwidth,clip]{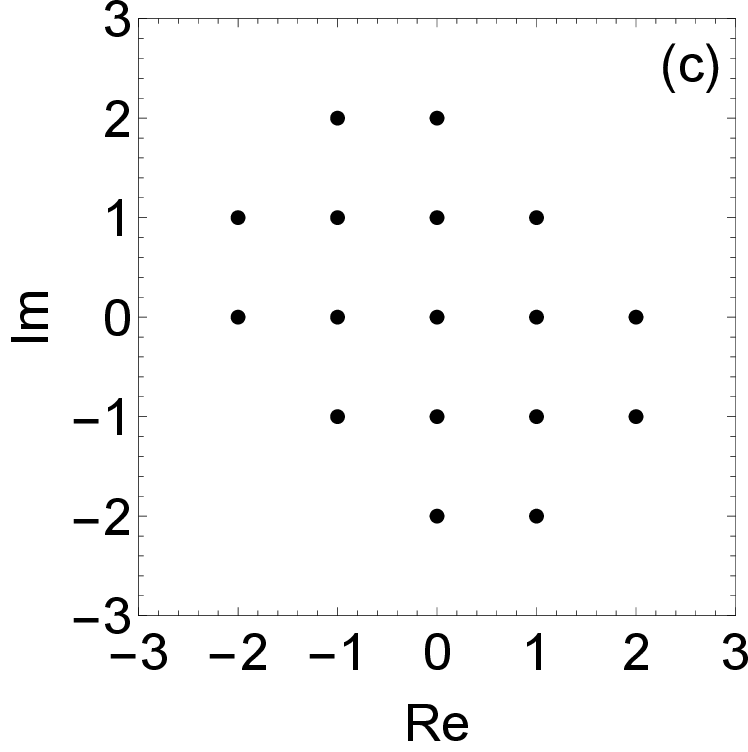}
\end{center}
\caption{\label{fg:G3}
(a) $G_3(1)=G_3(-1)$,
(b) $G_3\left(\sqrt{-1}\right)$, and
(c) $G_3\left(-\sqrt{-1}\right)$.}
\end{figure}

\begin{lem}\label{Lem:SpeedOfConvergence}
Let $u, t$  be complex numbers satisfying 
$u \in \left\{\pm 1,\pm \sqrt{-1}\right\}, t\in R_G\backslash G_3(u)$.
Then
\begin{align*}
\left|\frac{t\left(1-\sqrt{1-4u/t^2}\right)}{2}- F^{(n)}(0)\right| &<\frac{2}{ \rho^{2^{n+1}-1}},\\
\left|\frac{t\left(1+\sqrt{1-4u/t^2}\right)}{2}- F^{(n)}(t)\right| &<\frac{2}{ \rho^{2^{n+1}-1}}
\end{align*}
hold for all $n \ge 1$ with  $\rho =\frac{|t|+\sqrt{|t|^2-4}}{2}>1$.
\end{lem}

\begin{proof}
Remark that 
$|t| >2$.  Since
\begin{align*}
\left(h_n^2\sqrt{1-4u^{2^n}/h_n^2}\right)^2 &=h_n^4-4u^{2^n}h_n^2=\left(h_n^2-2u^{2^n}\right)^2-4u^{2^{n+1}}\\
&=h_{n+1}^2-4u^{2^{n+1}}=h_{n+1}^2\left(1-4u^{2^{n+1}}/h_{n+1}^2\right),
\end{align*}
$h_n^2\sqrt{1-4u^{2^n}/h_n^2}$ is equal to $h_{n+1}\sqrt{1-4u^{2^{n+1}}/h_{n+1}^2}$
or $-h_{n+1}\sqrt{1-4u^{2^{n+1}}/h_{n+1}^2}$.
We show that the latter is not suitable.
In fact, $\left| u^{2^n}/h_n^2 \right|<1/4$ (cf. \lemref{Lem:h_n__g_n}),
so that we have
\[
\left| \arg \frac{h_{n+1}}{h_n^2} \right|=
\left| \arg \left(1- \frac{2u^{2^n}}{h_n^2}\right) \right| < \frac{\pi}{6}
\quad \textrm{and} \quad
\left| \arg \sqrt{1- \frac{4u^{2^n}}{h_n^2}} \right| < \frac{\pi}{4}
\]
for any $n \ge 0$.
Thus,
\[
\left| \arg \left( \left(1- \frac{2u^{2^n}}{h_n^2}\right) 
\sqrt{1- \frac{4u^{2^{n+1}}}{h_{n+1}^2}} \left/
\sqrt{1- \frac{4u^{2^n}}{h_n^2}} \right. \right) \right| <
\frac{\pi}{6} + \frac{\pi}{4} + \frac{\pi}{4} < \pi,
\]
and we have $h_n^2\sqrt{1-4u^{2^n}/h_n^2} \neq -h_n^2 \left(1-2u^{2^n}/h_n^2\right)\sqrt{1-4u^{2^{n+1}}/h_{n+1}^2}$,
which implies $h_n^2\sqrt{1-4u^{2^n}/h_n^2} = h_{n+1}\sqrt{1-4u^{2^{n+1}}/h_{n+1}^2}$.
Hence, we get
\begin{align*}
\frac{h_n\left(1-\sqrt{1-4u^{2^n}/h_n^2}\right)}{2 h_0h_1\cdots h_{n-1}}-\frac{u^{2^n}}{h_0h_1\cdots h_n}
&=\frac{h_n^2-2u^{2^n}-h_n^2\sqrt{1-4u^{2^n}/h_n^2}}{2 h_0h_1\cdots h_n}\\
&=\frac{h_{n+1}-h_{n+1}\sqrt{1-4u^{2^{n+1}}/h_{n+1}^2}}{2 h_0h_1\cdots h_n},
\end{align*}
so that
\[
\frac{u^{2^n}}{h_0h_1\cdots h_n}= \frac{h_n-h_n\sqrt{1-4u^{2^n}/h_n^2 }}{2 h_0h_1\cdots h_{n-1}}-\frac{h_{n+1}-h_{n+1}\sqrt{1-4u^{2^{n+1}}/h_{n+1}^2}}{2
h_0h_1\cdots h_n},
\]
which together with  \lemref{Lem:NewtonSierpinski2} implies  
\begin{align*}
& F^{(n+1)}(0)\\
&=\sum_{0\le m\le n}    \frac{u^{2^m}}{h_0h_1\cdots h_m}\\
&=\sum_{0\le m\le n} \left( \frac{h_m-h_m\sqrt{1-4u^{2^m}/h_m^2 }}{2 h_0h_1\cdots
h_{m-1}}-\frac{h_{m+1}-h_{m+1}\sqrt{1-4u^{2^{m+1}}/h_{m+1}^2}}{2 h_0h_1\cdots h_m}\right)\\
&=\frac{t-t\sqrt{1-4u/t^2}}{2}-\frac{h_{n+1}-h_{n+1}\sqrt{1-4u^{2^{n+1}}/h_{n+1}^2}}{2 h_0h_1\cdots h_n}.
\end{align*}
Hence we get
\begin{align*}
&\left|\frac{t\left(1-\sqrt{1-4u/t^2}\right)}{2}- F^{(n+1)}(0)\right|
=\left|\frac{h_{n+1}\left(1-\sqrt{1-4u^{2^{n+1}}/h_{n+1}^2}\right)}{2
h_0h_1\cdots h_n}\right|\\
&=\left|\frac{2u^{2^{n+1}}}{h_0h_1\cdots h_nh_{n+1}\left(1+\sqrt{1-4u^{2^{n+1}}/h_{n+1}^2}\right)}\right|\\
&<\frac{2}{\rho^1\rho^2\cdots \rho^{2^n}\rho^{2^{n+1}}}
=\frac{2}{\rho^{2^{n+2}-1}},
\end{align*}
which implies 
\begin{align*}
\left|\frac{t\left(1+\sqrt{1-4u/t^2}\right)}{2}- F^{(n+1)}(t)\right|
&=\left|\frac{t\left(1-\sqrt{1-4u/t^2}\right)}{2}- F^{(n+1)}(0)\right|\\
&<\frac{2}{ \rho^{2^{n+2}-1}}
\end{align*}
since  $F^{(n+1)}(t)=t- F^{(n+1)}(0)$ by \lemref{Lem:NewtonSierpinski2}.
\end{proof}

\section{Main results}\label{MainResults}

Summing up Lemmas \ref{Lem:NewtonSierpinski2}, \ref{Lem:NewtonHurwitz}, we have the following

\begin{thm}\label{Thm:IdentitiesAsRationalFunctions}
Four identities
\begin{align*}
F^{(n+1)}(0) &=S_n(T,U) =\left[0;a_1,a_2,\ldots ,a_{2^{n+1}-1}\right],\\
F^{(n+1)}(T) &=T-S_n(T,U) =\left[b_0;b_1,b_2,\ldots ,b_{2^{n+1}-1}\right]
\end{align*}
are valid as rational functions in the field $C$, where $F^{(n)}(X)$
is the $n$-fold iteration of $F(X)$, $F(X)=F(X,f)$ is the Newton iterator for
$f(X)=X^2-T X+U$,
and
\begin{align*}
&S_n(T,U) := \sum_{0\le m\le n} \frac{U^{2^m}}{h_0(T) h_1(T,U)\cdots h_m(T,U)} \in  C,\\
&h_0 =T, \quad h_{n+1}=h_n^2-2U^{2^n} ~ (n\ge 0),\\
&a_n =U^{-1}T ~ (n=\mathrm{odd}\ge 1), \quad a_n=-T ~ (n=\mathrm{even}\ge 2),\\
&b_n =-U^{-1}T ~ (n=\mathrm{odd}\ge 1), \quad b_n=T ~ (n=\mathrm{even}\ge 0).    
\end{align*}
\end{thm}

In view of Lemma \ref{Lem:AbsoluteValuesOfRoots}, \defnref{Defn:PrincipalValueOfSquareRoot}, Lemmas  \ref{Lem:OpenDiscsIncludingRoots}, \ref{Lem:HCFGExpansions}  together with
\thmref{Thm:IdentitiesAsRationalFunctions},
we get the assertions (1)-(3) of \thmref{Thm:IdentitiesAsComplexNumbers}. The assertion (4) follows from \lemref{Lem:SpeedOfConvergence}.

\begin{thm}\label{Thm:IdentitiesAsComplexNumbers}
Let $G_2(u)$ be the finite subsets of
$R_G$ given in \lemref{Lem:OpenDiscsIncludingRoots}.   
Let  $u \in \left\{\pm 1, \pm \sqrt{-1}\right\}$, $t \in R_G \backslash G_2(u)$.
Let $F(X)$, $F^{(n)}(X)$, $S_n(T,U)$ be as in
\thmref{Thm:IdentitiesAsRationalFunctions}.
Then
\begin{enumerate}[(1)]
\item 
$f(X)=X^2-t X+u$ is irreducible over $K_G$ having two roots
$\alpha$, $\beta$ with $|\alpha |<1/2$, $|\beta -t|<1/2$,

\item
The Hurwitz continued fraction of $\alpha$, $\beta$ are given by
\[
\alpha =\left(HCF_G\right)\left[0;a_1,a_2,\ldots \right]
\]
with $a_n=u^{-1}t$ ($n=\mathrm{odd}\ge 1$), $a_n=-t$ ($n=\mathrm{even}\ge 2$);
\[
\beta =\left(HCF_G\right)\left[b_0;b_1,b_2,\ldots \right]
\]
with $b_n=-u^{-1}t$ ($n=\mathrm{odd}\ge 1$), $b_n=t$ ($n=\mathrm{even}\ge 0$).     

\item
Equalities
\begin{align*}
F^{(n+1)}(0) &=S_n(t,u)
=\left(HCF_G\right)\left[0;a_1,a_2,\ldots ,a_{2^{n+1}-1}\right],\\
F^{(n+1)}(t) &=t-S_n(t,u)
=\left(HCF_G\right)\left[b_0;b_1,b_2,\ldots ,b_{2^{n+1}-1}\right],
\end{align*}
hold for all $n\ge 0$.

\item
If in  particular, $u \in \left\{\pm 1,\pm \sqrt{-1}\right\}$, $t \in R_G \backslash G_3(u)$
with  $G_3(u)$
defined in Section \ref{SpeedConvergence}, then $\left|\alpha -F^{(n)}(0)\right|<\frac{2}{\rho^{2^{n+1}-1}}$, $\left|\beta
-F^{(n)}(t)\right| \allowbreak < \allowbreak \frac{2}{\rho^{2^{n+1}-1}}$
are valid for all $n \ge 1$, where $\rho =\frac{|t|+\sqrt{|t|^2-4}}{2}>1$.
\end{enumerate}
\end{thm}

\section{Relatively quadratic units over $K_E$}\label{EisensteinField}

We have seen
\thmref{Thm:IdentitiesAsComplexNumbers} related to relatively quadratic units over the Gaussian
field $K_G$
by introducing three finite subsets $G_1(u) \subset G_2(u) \subset
G_3(u)$ of $R_G$ and
by showing Lemmas \ref{Lem:AbsoluteValuesOfRoots}, \ref{Lem:OpenDiscsIncludingRoots} (\ref{Lem:HCFGExpansions}),
and \ref{Lem:SpeedOfConvergence} for $(t,u)\in R_G^2$ satisfying
\[
t \in R_G\backslash G_j(u), \quad u \in \left\{z \in \mathbb{C}; z^4=1\right\},
\]
respectively with $j = 1, 2, 3$.
Likewise, one can prove the following \thmref{Thm:IdentitiesAsComplexNumbersEisenstein} related to
relatively quadratic units over the Eisenstein field
$K_E$
by introducing three finite subsets $E_1(u) \subset E_2(u) \subset
E_3(u)$ of $R_E$ and
by showing three lemmas, say Lemmas
\ref{Lem:AbsoluteValuesOfRoots}',
\ref{Lem:OpenDiscsIncludingRoots}' (\ref{Lem:HCFGExpansions}'), and
\ref{Lem:SpeedOfConvergence}', parallel to Lemmas \ref{Lem:AbsoluteValuesOfRoots},
\ref{Lem:OpenDiscsIncludingRoots} (\ref{Lem:HCFGExpansions}), and
\ref{Lem:SpeedOfConvergence}, respectively for $(t,u)\in R_E^2$ satisfying
\[
t \in R_E\backslash E_j(u), \quad u \in \left\{z \in \mathbb{C}; z^6=1\right\}
\]
with $j = 1, 2, 3$,
where $E_j(u)$ ($j = 1, 2, 3$) are finite subsets of $R_E$ defined by \eqref{eq:E1} -- \eqref{eq:E3}
below:
\begin{equation}\label{eq:E1}
\begin{aligned}
E_1\left(b^0\right) &:=\{0,\pm 1,\pm 2\},\\
E_1\left(b^1\right) &:=\left\{x b+y \overset{\_}{b}; (x,y) =\pm (2,1) \right\} \cup \left\{0\right\},\\
E_1\left(b^2\right) &:=\left\{x b+y \overset{\_}{b}; (x,y) =\pm (1,0),\pm (2,0)\right\} \cup \left\{0\right\},\\
E_1\left(b^3\right) &:=\left\{x b+y \overset{\_}{b}; (x,y)=\pm (1,-1)\right\} \cup \left\{0\right\},\\
E_1\left(b^4\right) &:=\left\{x b+y \overset{\_}{b}; (x,y) =\pm (0,1),\pm (0,2)\right\} \cup \left\{0\right\},\\
E_1\left(b^5\right) &:=\left\{x b+y \overset{\_}{b}; (x,y)=\pm (1,2)\right\} \cup \left\{0\right\},
\end{aligned}
\end{equation}
where $b:=\frac{1+\sqrt{-3}}{2}$, $\overset{\_}{b}:=\frac{1-\sqrt{-3}}{2}$;

\begin{equation}\label{eq:E2}
\begin{aligned}
E_2\left(b^0\right) &:=\left\{x b+y \overset{\_}{b}; (x,y) \in \{\pm (1,2), \pm (2,1),\pm (2,2)\}\right\} \cup E,\\
E_2\left(b^1\right) &:=\left\{x b+y \overset{\_}{b}; (x,y) \in \{\pm (2,0), \pm (2,1),\pm (2,2)\}\right\} \cup E,\\
E_2\left(b^2\right) &:=\left\{x b+y \overset{\_}{b}; (x,y) \in \{\pm (1,-1), \pm (2,0), \pm (2,1)\} \right\} \cup E,\\
E_2\left(b^3\right) &:=\left\{x b+y \overset{\_}{b}; (x,y) \in \{\pm (0,2),\pm (1,-1),  \pm (2,0)\}\right\} \cup E,\\
E_2\left(b^4\right) &:=\left\{x b+y \overset{\_}{b}; (x,y) \in \{\pm (0,2), \pm (1,-1), \pm (1,2)\} \right\} \cup E,\\
E_2\left(b^5\right) &:=\left\{x b+y \overset{\_}{b}; (x,y) \in \{\pm (0,2), \pm (1,2), \pm (2,2)\} \right\} \cup E,
\end{aligned}
\end{equation}
where $E := \left\{0\right\} \cup \left\{b^j ; 0 \le j \le 5\right\}$,
see Figure~\ref{fg:E2};

\begin{figure}
\begin{center}
\includegraphics[width=0.4\textwidth,clip]{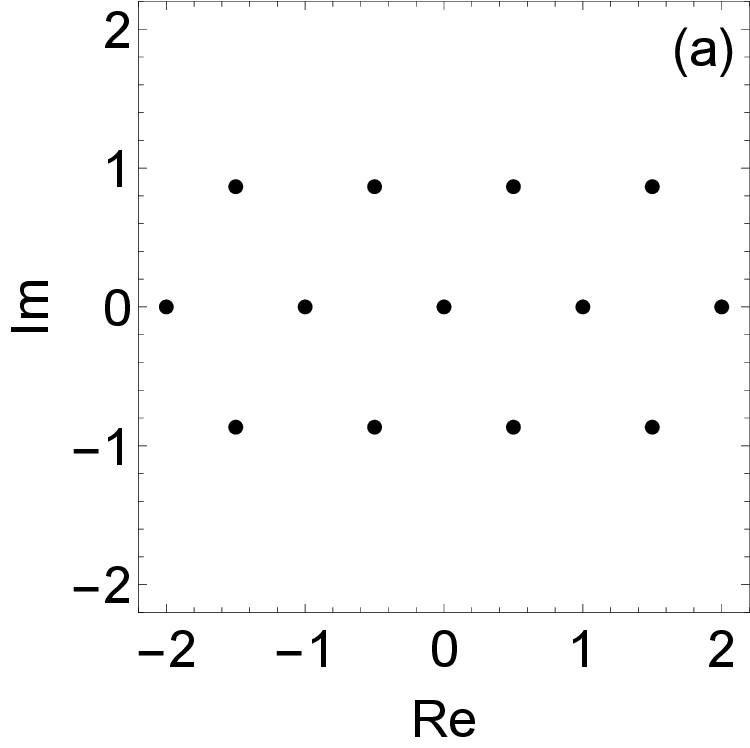}
\hspace*{0.05\textwidth}\includegraphics[width=0.4\textwidth,clip]{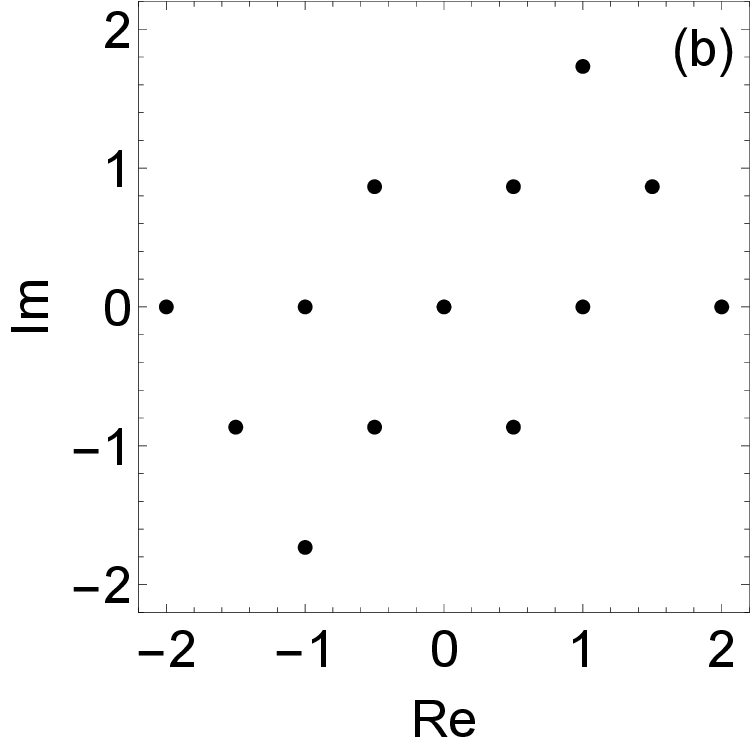}\\
\vspace*{0.025\textwidth}
\includegraphics[width=0.4\textwidth,clip]{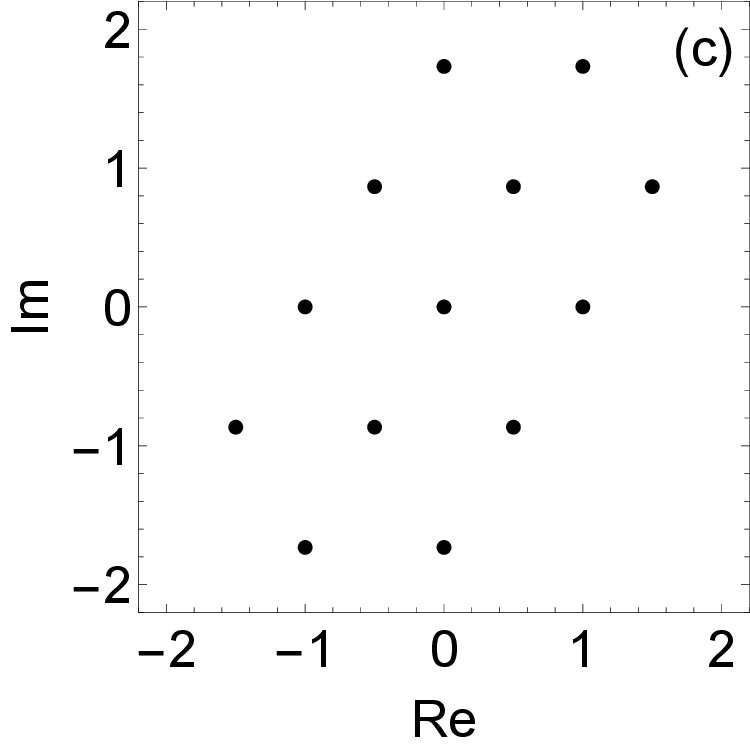}
\hspace*{0.05\textwidth}\includegraphics[width=0.4\textwidth,clip]{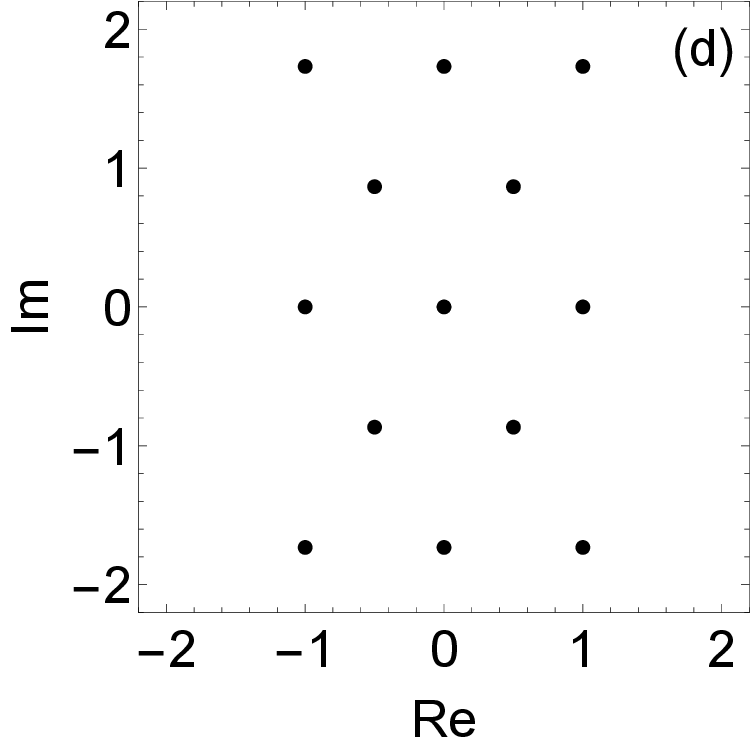}\\
\vspace*{0.025\textwidth}
\includegraphics[width=0.4\textwidth,clip]{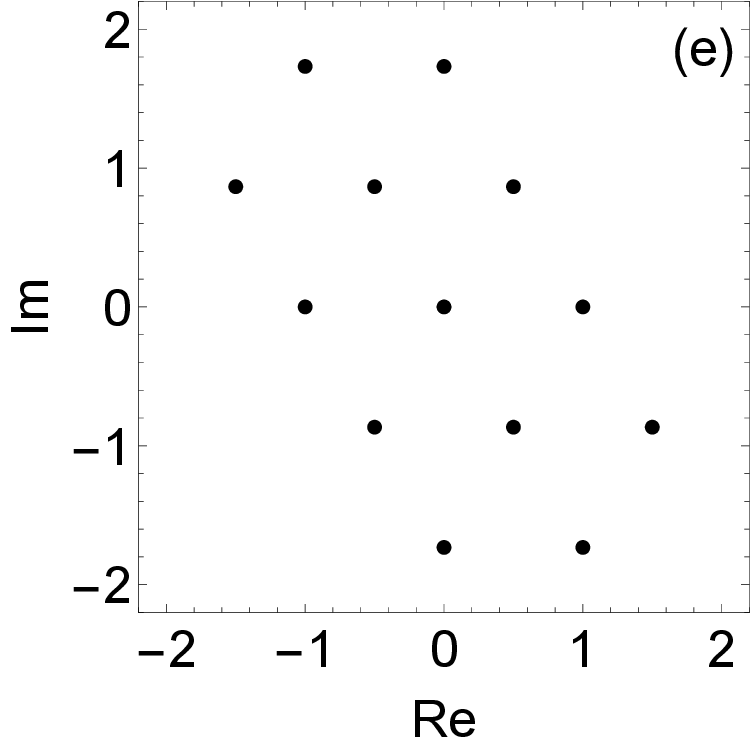}
\hspace*{0.05\textwidth}\includegraphics[width=0.4\textwidth,clip]{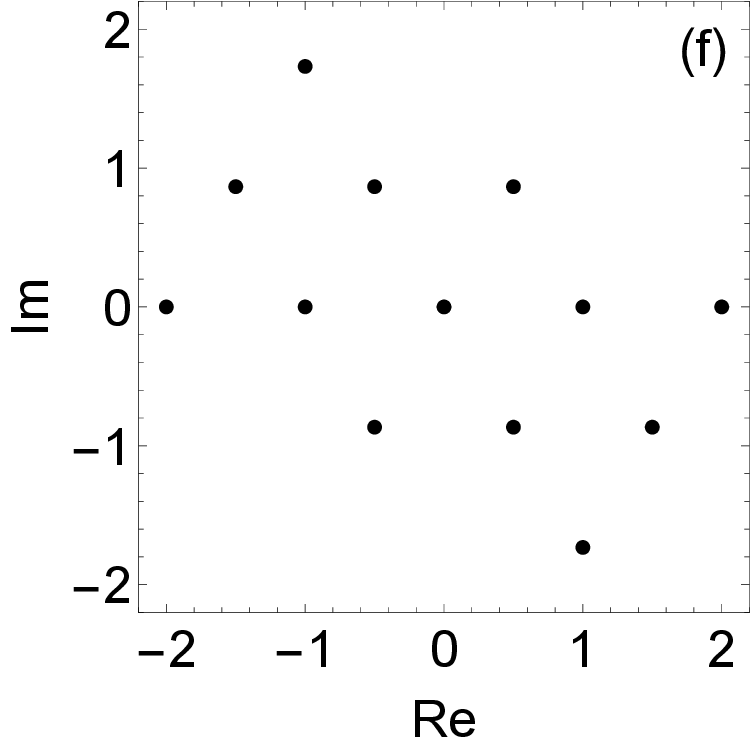}
\end{center}
\caption{\label{fg:E2}
(a) $E_2\left(b^0\right)$,
(b) $E_2\left(b^1\right)$,
(c) $E_2\left(b^2\right)$,
(d) $E_2\left(b^3\right)$,
(e) $E_2\left(b^4\right)$, and
(f) $E_2\left(b^5\right)$.}
\end{figure}

\begin{equation}\label{eq:E3}
E_3\left(u\right) := E_{*} \quad \left( \forall u \in \left\{z \in \mathbb{C}; z^6=1\right\} \right),
\end{equation}
where $E_{*} := \left\{z \in R_E; |z| \le 2\right\}$, see Figure~\ref{fg:E3}.

\begin{figure}
\begin{center}
\includegraphics[width=0.4\textwidth,clip]{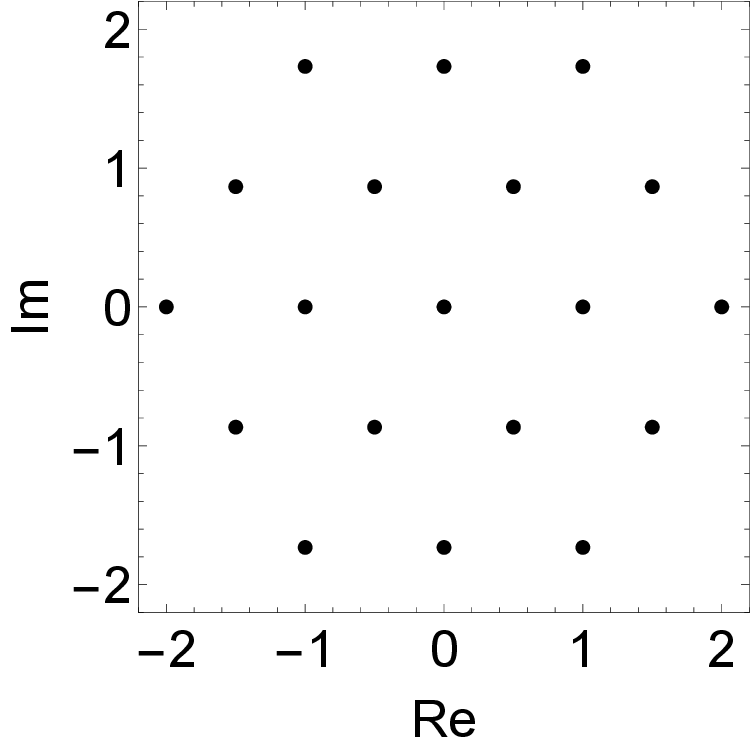}
\end{center}
\caption{\label{fg:E3}
$E_{*}$}
\end{figure}

Summing up Lemmas \ref{Lem:AbsoluteValuesOfRoots}',
\ref{Lem:OpenDiscsIncludingRoots}' (\ref{Lem:HCFGExpansions}'), and
\ref{Lem:SpeedOfConvergence}', we get

\begin{thm}\label{Thm:IdentitiesAsComplexNumbersEisenstein}
Let  $u \in \left\{z \in \mathbb{C}; z^6=1\right\}$, $t \in R_E\backslash
E_2(u)$ with $E_2(u)$
given by \eqref{eq:E2}.  Let  $F(X)$, $F^{(n)}(X)$,  $S_n(T,U)$ be as in
\thmref{Thm:IdentitiesAsRationalFunctions}. Then
\begin{enumerate}[(1)]
\item
$f(X)=X^2-t X+u$ is irreducible over $K_E$ having two roots 
$\alpha$, $\beta$ with $|\alpha |<1/2$, $|\beta -t|<1/2$.

\item
The Hurwitz continued fraction of $\alpha$, $\beta$ are given by
\[
\alpha =\left(HCF_E\right)\left[0;a_1,a_2,\ldots \right]
\]
with $a_n=u^{-1}t$ ($n=\mathrm{odd}\ge 1$), $a_n=-t$ ($n=\mathrm{even}\ge 2$);
\[
\beta = \left(HCF_E\right)\left[b_0;b_1,b_2,\ldots \right]
\]
with $b_n=-u^{-1}t$ ($n=\mathrm{odd}\ge 1$), $b_n=t$ ($n=\mathrm{even}\ge 0$).     

\item
Equalities
\begin{align*}
F^{(n+1)}(0) &=S_n(t,u)=\left(HCF_E\right)\left[0;a_1,a_2,\ldots ,a_{2^{n+1}-1}\right],\\
F^{(n+1)}(t) &=t-S_n(t,u)=\left(HCF_E\right)\left[b_0;b_1,b_2,\ldots ,b_{2^{n+1}-1}\right],
\end{align*}
hold for all $n\ge 0$.

\item
If in  particular, $u \in \left\{z\in \mathbb{C}; z^6=1\right\}$, $t \in R_E\backslash E_3(u)$ with $E_3(u)$
given by \eqref{eq:E3},
then $\left|\alpha -F^{(n)}(0)\right| < \frac{2}{\rho^{2^{n+1}-1}}$, $\left|\beta -F^{(n)}(t)\right|<\frac{2}{\rho^{2^{n+1}-1}}$  are 
valid for all $n \ge 1$, where $\rho =\frac{|t|+\sqrt{|t|^2-4}}{2}>1$.
\end{enumerate}
\end{thm}


\subsection*{Acknowledgment}
This research was supported by JSPS KAKENHI Grant Numbers
JP22K12197 and JP25K15277.

\end{document}